\documentclass[11pt,a4paper]{amsart}
\usepackage{amssymb}
\usepackage{enumerate}
\usepackage{mathtools}
\usepackage{xcolor}
\usepackage{textcomp}
\usepackage{amsthm}
\usepackage{array}
\usepackage[margin=2.8cm]{geometry}
\usepackage{float}
\usepackage{hyperref}
\usepackage{rotating}

\newcolumntype{L}{>{$}l<{$}}

 \theoremstyle{plain}
\theoremstyle{definition}

\newtheorem{definition}{Definition}[section]

\theoremstyle{remark}
\newtheorem{remark}[definition]{Remark}

\newtheorem{corollary}[definition]{Corollary}

\newtheorem{lemma}[definition]{Lemma}

\newtheorem{theorem}[definition]{Theorem}
\newtheorem{proposition}[definition]{Proposition}

\newcounter{claim}[definition]

\theoremstyle{remark}

\theoremstyle{remark}

\theoremstyle{definition}

\def\cn{\mathord{{\!\:{:}\:\!}}}

\newcommand{\ns}{\mathbin{\mspace{-4mu}\vcenter{\hbox{\raisebox{1ex}{$\cdot$}}}\mspace{-4mu}}}
\newcommand{\twist}[1]{%
  {}^{2}\!#1%
}

\newcommand{\GG}{\mathcal{G}}

\newcommand{\BM}{\mathbb{B}}
\newcommand{\MM}{\mathbb{M}}

\DeclareMathOperator{\Fix}{Fix}

\DeclareMathOperator{\Syl}{Syl}

\DeclareMathOperator{\Ly}{Ly}
\DeclareMathOperator{\McL}{McL}
\DeclareMathOperator{\M}{M}
\DeclareMathOperator{\HN}{HN}
\DeclareMathOperator{\HS}{HS}
\DeclareMathOperator{\He}{He}
\DeclareMathOperator{\J}{J}
\DeclareMathOperator{\Suz}{Suz}
\DeclareMathOperator{\Th}{Th}
\DeclareMathOperator{\Co}{Co}
\DeclareMathOperator{\ON}{O'N}
\DeclareMathOperator{\Ru}{Ru}
\DeclareMathOperator{\Fi}{Fi}

\DeclareMathOperator{\B}{B}

\DeclareMathOperator{\E}{E}
\DeclareMathOperator{\F}{F}
\DeclareMathOperator{\G}{G}

\DeclareMathOperator{\Alt}{Alt}

\DeclareMathOperator{\Sym}{Sym}

\DeclareMathOperator{\Irr}{Irr}

\renewcommand{\epsilon}{\varepsilon}

\DeclarePairedDelimiter{\gen}{\langle}{\rangle}

\title[Gill--Guillot graph for sporadic groups]{The Gill--Guillot commuting graph \\for sporadic and related groups\vspace*{-0.4cm}}
\author{David A. Craven}
\address{School of Mathematics, University of Birmingham, UK}
\email{d.a.craven@bham.ac.uk}
\author{Coen del Valle}\address{School of Mathematics and Statistics,
The Open University, UK
}
\email{Coen.del-Valle@open.ac.uk}
\author{Chris Parker}\address{School of Mathematics, University of Birmingham, UK
}
\email{c.w.parker@bham.ac.uk}
\date{\today}
\thanks{We are grateful to the anonymous referee for their helpful comments which have greatly improved this paper. Research of Coen del Valle is supported by the Engineering and Physical Sciences Research Council
[grant number EP/Z534742/1].}
\begin{document}
\begin{abstract}\vspace*{-0.3cm} Let $G$ be a finite group and $\mathcal{C}$ a normal subset of $G$. The \emph{Gill--Guillot graph} has vertex set $\mathcal C$ with distinct $x, y \in \mathcal C$ adjacent if and only if $x$ and $y$ commute and $\{xy^{-1},x^{-1}y\} \cap \mathcal C$ is non-empty. We study the connectivity of this graph for quasisimple groups with $G/Z(G)$ a sporadic simple group and for certain simple groups with exceptional Schur multiplier.
\end{abstract}
\vspace*{-0.9cm}\maketitle\vspace*{-0.5cm}
\vspace*{-0.25cm}

\section{Introduction}
Motivated by applications to the study of relational structures \cite{cherlin2000,Lach84} and binary actions, in \cite{GG25} Gill and Guillot introduced the \emph{Gill--Guillot graph} of a finite group $G$, which we will denote by $\GG(\mathcal C)$. This is a graph with vertex set a normal subset $\mathcal C$ of $G$, in which distinct vertices $x, y \in \mathcal C$ are adjacent if and only if $x$ and $y$ commute and $xy^{-1}$ or its inverse lies in $\mathcal C$---this collapses to the single condition $xy \in \mathcal C$ when $\mathcal C$ consists of only involutions. Gill and Guillot used connectivity properties of $\GG(\mathcal C)$ to deduce structural information about point stabilizers in binary actions.

The graph $\GG(\mathcal C)$ is a subgraph of a commonly used graph in group theory, the \emph{commuting graph} on $\mathcal{C}$, in which distinct elements of $\mathcal{C}$ are adjacent if they commute. More generally, for a finite group $G$ and subset $X\subseteq G$, the \emph{commuting graph}, denoted $\mathrm{Com}(G,X)$, is the graph on $X$ in which $x,y\in X$ are adjacent if and only if $[x,y]=1$. When $X$ is the set of elements of order $p$, this graph is disconnected if and only if the group possesses a strongly $p$-embedded subgroup (see Proposition~\ref{prop:spe}), so the question of whether $\GG(\mathcal C)$ is connected for a conjugacy class of $p$-elements $\mathcal C$ of a finite group $G$ is related to the question of whether $G$ has a strongly $p$-embedded subgroup. Commuting graphs have played an important role in the local analysis of finite groups; for example, the Brauer--Fowler theorem was proved by analysing commuting involutions, and commuting graphs have also been used in recognition and uniqueness results for finite simple groups (see, for example, \cite{brauerfowler1955,AschSeg92}).

As an example close to our specific graph, in \cite[Theorem 1]{AschbacherExistence} Aschbacher considers $\mathrm{Com}(G,\mathcal D)$ in the case when $\mathcal D$ is a normal subset of involutions such that for all edges $\{u,v\}$ in $\mathrm{Com}(G,\mathcal D)$, the product $uv$ lies in $\mathcal D$. Hence in this case $\mathrm{Com}(G,\mathcal D)=\GG(\mathcal D)$.
Aschbacher proves that if $\GG(\mathcal D)$ is disconnected, then for $H$ the stabilizer of a connected component of $\GG(\mathcal D)$, we have that $H \cap \gen{\mathcal D}$ is a strongly embedded subgroup of $\gen{\mathcal D}$. Hence, if $G$ is a finite simple group, then using \cite[Satz 4]{Bender} we know that $G \cong \mathrm{L}_2(2^a)$, $\mathrm{U}_3(2^a)$, or $\twist\B_2(2^{2a-1})$ with $a\ge 2$. Aside from their many applications, commuting graphs have been extensively studied in their own right, with a significant body of work dedicated to understanding their combinatorial properties, such as connectivity (see, for example,~\cite{BBHR07,IranJafar08,MorPar13}).

Inspired by \cite{AschbacherExistence} and applications to binary actions, it is natural to determine the connectivity of $\GG(\mathcal C)$ in the situation where $G=\gen{\mathcal C}$ (a condition which -- when $\GG(\mathcal{C})$ is not edgeless and $\mathcal C$ is a single conjugacy class -- forces $G$ to be perfect) and this is the setting in which our work sits. In the case that $\mathcal C$ is a conjugacy class of involutions, the connectivity of $\GG(\mathcal C)$ is settled for $G$ an alternating group in \cite[Proposition 3.4]{GG25}, for $G$ a simple group of Lie type in characteristic $2$ in \cite[Theorem 2]{GGL25}, and in the case where $G$ is simple and $\mathcal C$ is the only involution class in \cite[Proposition 3.1]{GG252}. The case in which $G$ is a quasisimple group of Lie type of untwisted rank at most 2, and $p$ is either $2$ or the defining characteristic of $G$, is treated in work currently in preparation \cite{dVGL26}.

In this article we determine the connectivity of $\GG(\mathcal C)$ for conjugacy classes $\mathcal C$ of elements of any prime order $p$, when $G$ is a quasisimple group with $G/Z(G)$ a sporadic simple group and when $G$ is an exceptional cover of either an alternating group or a simple group of Lie type. In the interest of brevity, our main results will be restricted to the case $Z(G)$ is a $p$-group; this is sufficient to understand the general case (see Lemma~\ref{lem:pnotcenter} and the discussion following it). For completeness we shall consider ${}^2\mathrm{F}_4(2)'$ as a sporadic group.

Suppose that $\GG=\GG(\mathcal{C})$ is a Gill--Guillot graph and let $t \in \mathcal C$. Then $\Lambda(t)$ denotes the connected component of $\GG$ containing $t$ and the \emph{component group} of $t$ is $\Delta(t)=\gen{\Lambda(t)}$. Throughout, conjugacy classes are named according to the {\sf GAP} Character Table Library~\cite{CTblLib1.3.11}. We can now summarize our main results as follows. 

\begin{theorem}\label{thm:mainresults}
Let $p$ be a prime, let $G$ be a quasisimple group with $Z(G)$ a (possibly trivial) $p$-group, and let $t\in G\setminus Z(G)$ have order $p$. Set
\[ \GG=\GG(t^G),\;\;\;\; \Lambda=\Lambda(t),\;\;\;\; \Delta=\Delta(t).\]
Assume either that $G/Z(G)$ is a sporadic simple group, or that $G$ is an exceptional cover of either an alternating group or a simple group of Lie type. If $\GG$ is disconnected, then one of the following holds:
\begin{enumerate}
\item[(a)] $G$ is a sporadic simple group and either
\begin{enumerate}
\item[(i)] $\Delta=\langle t\rangle$; 
\item[(ii)] $\Delta$ is elementary abelian of order $p^2$; or
\item[(iii)] $\Delta\cong \J_2,\, \HS,\, \mathrm{O}_8^+(2)$, or $5_+^{1+4}$.
\end{enumerate}
In each case, if $p^2\mid |G|$ then the pair $(G,t^{G})$ is explicitly listed, otherwise $\Delta=\langle t\rangle$.

\item[(b)] $G/Z(G)$ is a sporadic simple group, $Z(G) $ is a non-trivial $p$-group, and either
\begin{enumerate}
\item[(i)] $p=2$ and $(G,t^G)$ is one of $(2\ns\J_2,\mathrm{2c})$, $(2\ns\Fi_{22},\mathrm{2b})$, $(2\ns\Fi_{22},\mathrm{2c})$, $(2\ns\BM,\mathrm{2b})$, in each case $\GG$ is edgeless; or
\item[(ii)] $p=3$ and $\Delta$ is elementary abelian, with the possibilities given in Table~\ref{tbl:3covers}.
\end{enumerate}

\item[(c)] $G/Z(G)$ is either a simple group of Lie type in characteristic $p$ with exceptional Schur multiplier, or $G/Z(G)\cong \Alt(6)$ or $\Alt(7)$ with $p=3$, and the pairs $(G,t^G)$ for which $\GG$ is disconnected are listed in Table~\ref{tbl:exco}.
\end{enumerate}
\end{theorem}

Returning to the initial motivation, the recent article~\cite{delvallegill2025+} uses the following corollary, which follows from \cite[Corollary~2.13]{GG25} together with the explicit information in Propositions~\ref{prop:explicit1},~\ref{prop:explicit2}, and~\ref{prop:explicit3}.

\begin{corollary}
Suppose that $p$ is a prime dividing $|G|$ and $ (G,p)$ appears in Table~\ref{tbl:forbiddenpbinary}. Then every transitive binary action of $G$ has point stabilizer of order coprime to $p$. In particular, in such an action the Sylow $p$-subgroup of $G$ is semi-regular.
\end{corollary}

\begin{table}[ht]
\centering
\begin{tabular}{|ll|ll|ll|ll|}\hline
$G$            & $p$   & $G$       & $p$     & $G$        & $p$       & $G$         & $p$             \\ \hline
$\M_{11}$      & $2,3$ & $\HS$     & $2,3$   & $\ON$      & $2,3$     & $\Fi_{23}$  & $3,5$           \\
$\M_{12}$      & $2,3$ & $\J_3$    & $2$     & $\Co_3$    & $2,3$     & $\Co_1$     & $2,5,7$         \\
$\J_1$         & $2$   & $\M_{24}$ & $2,3$   & $\Co_2$    & --        & $\J_4$      & $2,3$           \\
$\M_{22}$      & $2,3$ & $\McL$    & $2$     & $\Fi_{22}$ & $3,5$     & $\Fi_{24}'$ & $2,3,5,7$       \\
$\J_2$         & $2$   & $\He$     & $2,3,5$ & $\HN$      & $2,3$     & $\BM$       & $3,5,7$         \\
$\M_{23}$      & $2,3$ & $\Ru$     & $2,3,5$ & $\Ly$      & $2$       & $\MM$       & $2,3,5,7,11,13$ \\
$\twist\F_4(2)'$ & $2,3$ & $\Suz$    & $2$     & $\Th$      & $2,3,5,7$ &             &   \\             \hline
\end{tabular}
\caption{Sporadic groups and primes that do not divide the stabilizer of a transitive binary action.}
\label{tbl:forbiddenpbinary}
\end{table}

The groups in Theorem~\ref{thm:mainresults} exhibit genuinely exceptional behaviour, so our arguments necessarily use specific information about the groups and thus the approach in this paper is distinct from that needed for the uniform infinite families of groups. We prove our results using a mixture of computational and theoretical techniques. The smallest of the sporadic groups may be dealt with directly, using computer code appearing in the supplementary file~\cite{codesup}, and thus much of this paper is dedicated to justifying the techniques which the code uses, and then explicitly using such techniques for the structural arguments required for the larger sporadic groups.

The structure of the paper is straightforward. In Section~\ref{Setup}, we briefly introduce the necessary definitions and notation, before stating and proving the technical lemmas which are used in proving the main results. We emphasize here that the results in Section~\ref{Setup} and the computer code in the supplementary materials provide a solid foundation for analysing Gill--Guillot graphs for groups other than those appearing in this paper. Section~\ref{s:simple} is dedicated to proving Theorem~\ref{thm:mainresults}(a), and Section~\ref{s:covers} handles the central extensions as listed in Theorem~\ref{thm:mainresults}(b) and (c).

\section{Setup and preliminary lemmas}\label{Setup}

Suppose that $G$ is a finite group and $p$ is a prime dividing $|G|$. Let $\mathcal C$ be a union of conjugacy classes (i.e., a normal subset) of $G$. As in the introduction, we define the Gill--Guillot graph $\GG(\mathcal C)$ to be the graph with vertex set $\mathcal C$ and distinct vertices $u$ and $v$ joined by an edge if and only if $u$ and $v$ commute and at least one of $uv^{-1}$ or $u^{-1}v$ lies in $\mathcal C$. Notice that the conjugation action of $G$ on the vertices of $\GG(\mathcal C)$ preserves the edge set and so induces a map from $G$ into the automorphism group of $\GG(\mathcal C)$. Let $t \in G$ have order $p$.

We will investigate two cases: the first case is $\mathcal C=t^G$. The second is where $\mathcal C$ is the \emph{rational class} of $t$, consisting of all non-identity powers of elements of $t^G$, and we write $\overline{t^G}$ for this set. Hence $ \overline{t^G}=\bigcup_{g\in G} (\gen{t}\setminus \{1\})^g$.
We generalize this to normal subsets $\mathcal C$ of $G$ to obtain the \emph{rational closure} $\overline {\mathcal C}$, which contains all the elements $y \in G$ with $\gen{y} =\gen{x}$ for some $x \in \mathcal C$.

For $t \in \mathcal C$, the connected component of $\GG(\mathcal C)$ containing $t$ is denoted by $\Lambda(t)$.
 The subgroup $\Delta(t):=\langle \Lambda(t)\rangle$ of $G$ is the \emph{component group} of $t$ in $\GG(\mathcal C)$, and the \emph{stabilizer} $N_G(\Lambda(t))$ of $\Lambda(t)$ is the group
\[ \{g\in G\mid t^g\in \Lambda(t)\}\le N_G(\Delta(t)).\]
We implicitly use that if $s\in \Lambda(t)$ then $\Lambda(s)=\Lambda(t)$ and if $g \in G$ then $\Lambda(t)^g=\Lambda(t^g)$. Thus, $N_G(\Lambda(t))$ depends only on the component $\Lambda(t)$ and not on the choice of $t$ in this component.

\begin{lemma}\label{lem:rat1} Suppose $\mathcal C$ is a normal subset of $G$. If $t \in \mathcal C$ then $\gen{t}\setminus \{1\}$ lies in the connected component of $\GG(\overline{\mathcal C})$ which contains $t$.
\end{lemma}

\begin{proof} We can assume that $t$ does not have order $2$. Let $s \in \gen{t}\setminus \{t,1\}$. Then $st^{-1} \in \gen{t}\setminus \{1\} \subseteq \overline{\mathcal C}$. Hence $s$ is adjacent to $t$ in $\GG(\overline{\mathcal C})$. This proves the claim.\end{proof}

Observe that $\Delta(t)\leq N_G(\Lambda(t))$: indeed, each generator $s$ of $\Delta(t)$ is a vertex of $\Lambda(t)$ which fixes itself and thus stabilizes the whole component.

Since we are interested in the connectivity of $\GG(\mathcal C)$, the importance of the following lemma is evident. It applies in particular in the case where $\mathcal C$ is a conjugacy class or a rational class. (That this applies to rational classes follows from Lemma~\ref{lem:rat1}.)

\begin{lemma}\label{lem:connectedstabilizer} Let $\mathcal C=\bigsqcup_{i=1}^m\mathcal{C}_i$ be a union of conjugacy classes of elements of $G$. Let $\Lambda$ be a connected component of $\GG(\mathcal C)$, and suppose that $\Lambda\cap\mathcal{C}_i\ne\emptyset$ for all $1\leq i\leq m$. If $N_G(\Lambda)=G$ or $\Delta:=\langle\Lambda\rangle=G$ then $\GG(\mathcal C)$ is connected.
\end{lemma}
\begin{proof} We prove the result for $N_G(\Lambda)$; the result for $\Delta$ follows from the discussion above. For each $i$, pick some $t_i\in\Lambda\cap\mathcal{C}_i$. Since $G$ stabilizes $\Lambda$, it follows that $t_i^G \subseteq \Lambda$, whence $\mathcal{C}=\bigsqcup_{i=1}^mt_i^G\subseteq \Lambda$, as desired.
\end{proof}

We repeatedly use the fact that if $t$ is a vertex in $\GG(\mathcal C)$ then $C_G(t)$ leaves $\Lambda(t)$ invariant.

\begin{lemma}\label{lem:componentgroup} Suppose that $t\in G$, and set $\mathcal C=t^G$ and $H=N_G(\Lambda(t))$. Then $|\Lambda(t)|={\nolinebreak|H:C_G(t)|}$ and $\Lambda(t)=t^{H}$. In particular, $\Delta(t)=\gen{t^H}$.
\end{lemma}
\begin{proof}
It suffices to show that $\Lambda(t)=t^H$, from which the rest follows from the Orbit-Stabilizer theorem. First, suppose that $u\in \Lambda(t)$. Then there exists $g\in G$ such that $t^g=u$. Hence $\Lambda(t)^g=\Lambda(t^g)= \Lambda(u)= \Lambda(t)$ and so $g\in H$ and $u\in t^H$. Conversely, if $u\in t^H$ then there exists $h\in H$ such that $u=t^h\in \Lambda(t)^h=\Lambda(t)$.
\end{proof}

Another observation is that for $\mathcal C$ a normal subset of $G$ and $t\in \mathcal C$, the vertices adjacent to $t$ are in $C_G(t)$. Hence $C_G(t)$ permutes these vertices by conjugation. The next two lemmas are stated in more generality than we shall need. For both we fix the following set-up: let $\mathcal C=t^G$ be a conjugacy class in $G$, let $\Gamma$ be a subgraph of $\mathrm{Com}(G,\mathcal{C})$ on which $G$ acts by automorphisms (as is the case for $\GG(t^G)$), and define $\Sigma$ to be the connected component of $\Gamma$ containing $t$. As for the Gill--Guillot graph, we shall write $N_G(\Sigma)$ to denote the stabilizer of the component $\Sigma$.

\begin{lemma}\label{lem1} Suppose that $u_1,\dots,u_r$ are representatives of the $C_G(t)$-orbits on the vertices of $\Gamma$ adjacent to $t$. Let $g_1,\dots,g_r$ be elements of $G$ such that $t^{g_i}=u_i$. Then $N_G(\Sigma)= \gen{C_G(t),g_1,\dots,g_r}$.
\end{lemma}
\begin{proof}
That $L:=\langle C_G(t),g_1,\dots,g_r\rangle\leq N_G(\Sigma)$ is clear; suppose for contradiction that the containment is proper. Define $d(x,y)$ to be the length of a shortest path between $x$ and $y$ in $\Sigma$. Choose $g \in N_G(\Sigma)\setminus L$ with $d(t,t^g)$ minimal, and let $t=v_0,v_1,\dots, v_n=t^g$ be a shortest path in $\Sigma$. Then there exists $c\in C_G(t)$ such that $v_1^c= u_i$ for some $i$ and so $v_1^{c}=t^{g_i}$. Thus
\[ n-1=d(v_1,t^g)=d(t^{g_ic^{-1}},t^g)=d(t,t^{gcg_i^{-1}}),\]
and so $gcg_i^{-1} \in L$ by the choice of $g$. But then $g\in L$, a contradiction. Hence $L =N_G(\Sigma)$.
\end{proof}

When we compute in the quasisimple group $3\ns\Fi_{24}'$, to decrease the time taken for the computation it is useful to be able to calculate orbit representatives in a smaller subgroup than $C_G(t)$. Hence we need a modest variant of Lemma~\ref{lem1}.

\begin{lemma}\label{lem1variant} Let $p$ be a prime and suppose that $t\ne 1 $ is a $p$-element, $S \in \Syl_p(C_G(t))$ and $ w_1, \dots, w_\ell $ are representatives of the $S$-orbits on the neighbours of $t$ in $\Gamma$ contained in $S$. Assume that, for $1 \le i \le \ell$, $h_i \in G$ is such that $t^{h_i}=w_i$. Then the stabilizer of $\Sigma$ is $\gen{C_G(t), h_1, \dots,h_\ell}$.
\end{lemma}
\begin{proof} Set $L=\gen{C_G(t), h_1, \dots,h_\ell}$. Obviously $L$ stabilizes $\Sigma$. Let $u_1,\dots,u_r$ be $C_G(t)$-orbit representatives of the vertices adjacent to $t$ and fix some $u=u_i$. Then $u \in C_G(t)$ and $\gen{t,u} $ is a $p$-group. By Sylow's Theorem, there exists $g$ in $C_G(t)$ such that $\gen{t,u}^g= \gen{t,u^g} \le S$.

Since $u^g$ is adjacent to $t$, there exists $k \in S$ such that $u^{gk}=w_j$ for some $1\le j \le \ell$. By assumption, there exists $h_j$ in $L$ such that $t^{h_j}= w_{j}$. Hence $t^{h_jk^{-1}g^{-1}}=u=u_i$. Since $k^{-1}g^{-1} \in C_G(t)$, we deduce that $h_jk^{-1}g^{-1} \in L$. But then for each $1\le i\le r$, $L$ contains an element $\gamma_i=h_jk^{-1}g^{-1}$ such that $t^{\gamma_i}= u_i$. Thus $L$ contains the stabilizer of $\Sigma$ by Lemma~\ref{lem1}. As $L$ stabilizes $\Sigma$, this concludes the proof.
\end{proof}

\begin{lemma}\label{lem:notcyclic} Suppose that $t\in G$ has order $p$ and define $\mathcal C=t^G$. If the Sylow $p$-subgroups of $G$ are cyclic or generalized quaternion (so $p=2$ in the latter case) then $\Lambda(t)\subset \langle t\rangle$ and $|\Lambda(t)|\le p-1$.
\end{lemma}
\begin{proof} Let $S$ be a Sylow $p$-subgroup containing $t$. Since $S$ is cyclic or generalized quaternion, $\langle t\rangle$ is the unique cyclic subgroup of $S$ of order $p$. Therefore, the only edges in $\Lambda(t)$ are between elements in $\langle t\rangle$. Hence $\Lambda(t) \subset \langle t\rangle$ and $|\Lambda(t)|\le p-1$.
\end{proof}

We also present a proof of the result mentioned in the introduction which relates connectivity of certain Gill--Guillot graphs to strongly $p$-embedded subgroups. Recall that a proper subgroup $H$ of a finite group $G$ is called \emph{strongly $p$-embedded} if $p$ is a prime which divides the order of $H$ but does not divide the order of $H\cap H^g$ for any $g\in G\setminus H$. A subgroup is \emph{strongly embedded} if it is strongly 2-embedded.

\begin{lemma}\label{lem:prophelp} Suppose that $G$ has a strongly $p$-embedded subgroup and that $\mathcal C$ is a normal subset of non-trivial $p$-elements. Then both $\mathrm{Com}(G,\mathcal C)$ and $\GG(\mathcal C)$ are disconnected.
\end{lemma}
\begin{proof} Since $\GG(\mathcal C)$ is a subgraph of $\mathrm{Com}(G,\mathcal C)$ it suffices to show that $\mathrm{Com}(G,\mathcal C)$ is disconnected. Suppose that $L$ is strongly $p$-embedded in $G$. Then $L$ contains a Sylow $p$-subgroup and so $\mathcal{C} \cap L\ne\emptyset$. Let $t \in \mathcal C \cap L$ and define $\Lambda$ to be the connected component of $\mathrm{Com}(G,\mathcal C)$ which contains $t$. We claim that $\Lambda \subset L$. Suppose that $s\in \Lambda \cap L$ and $r\in \Lambda\setminus L$ is adjacent to $s$. Then $s=s^r \in L \cap L^r$. Since $L$ is strongly $p$-embedded and $\gen{s}$ is a non-trivial $p$-group, we have $r \in L$, a contradiction. Therefore $\Lambda \subset L$.

Assume that $\mathrm{Com}(G,\mathcal C)$ is connected. Then $\mathcal C=\Lambda$. But then for all $g\in G$, $L\cap L^g \ge \gen{\mathcal C}$ contradicting $L$ being strongly $p$-embedded in $G$. This proves the result.
\end{proof}

\begin{proposition}\label{prop:spe} Suppose that $\mathcal C$ is the set of elements of $G$ of order $p$. Then $\GG(\mathcal C)$ is disconnected if and only if $G$ has a strongly $p$-embedded subgroup. Moreover, the pairs $(G,p)$ where $G$ is a finite simple group and $p$ is a prime such that $G$ has non-cyclic Sylow $p$-subgroups and a strongly $p$-embedded subgroup, are listed in Table~\ref{tbl:spe}.
\end{proposition}
\begin{proof} First observe that if $\{x,y\}$ is an edge in $\mathrm{Com}(G,\mathcal{C})$, then $x$ and $y^{-1}$ commute, whence $xy^{-1}$ has order $p$. That is, $xy^{-1}\in\mathcal{C}$. Thus $\GG(\mathcal C)=\mathrm{Com}(G,\mathcal{C})$. Now, let $S \in \mathrm{Syl}_p(G)$, and pick $t \in \mathcal C \cap Z(S)$. Then $t$ is connected to every $s\in \mathcal C \cap S$ distinct from $t$; it follows that $N_G(S)$ stabilizes $\Lambda (t)$, as does $C_G(s)$. Write $M=N_G(\Lambda(t))$.

Suppose that $\GG(\mathcal{C})$ is disconnected. Then $M<G$ by Lemma~\ref{lem:connectedstabilizer}. On the other hand, we have seen that $\langle C_G(s),N_G(S)\mid s \in \mathcal C \cap S\rangle\leq M$, so by~\cite[Proposition 17.11]{GLS2}, $M$ is strongly $p$-embedded in $G$. If $G$ has a strongly $p$-embedded subgroup then Lemma~\ref{lem:prophelp} implies $\GG(\mathcal C)$ is disconnected.

The data of Table~\ref{tbl:spe} comes from \cite[Theorem 7.6.1]{GLS3}.
\end{proof}

\begin{table}[H]
   \centering
     \begin{tabular}{|lc|lc|lc|}
\hline
$G $                           & $ p $    & $ G $                          & $ p $    & $ G $                  & $ p $ \\\hline
$\mathrm{L}_2(r^{n+1})$ & $r$ & $\mathrm U_3(r^n)$, $r^n\ne 2$ & $r$ & $\twist\B_2(2^{2n+1})$& $2$ \\
$\twist\G_2(3^{2n+1})$           & $3$    & $\Alt(2r)$, $r\ne 2$                     & $r$ &  & \\ \hline
$\mathrm L_3(4)$              & $3$ &$\M_{11} $                     & $ 3$     & $ \twist\F_4(2)' $             & $ 5 $ \\     $ \McL $               & $5 $  &
$ \Fi_{22} $                   & $5 $     & $ \J_4$                        & $11$                               \\ \hline
\end{tabular}
\caption{Simple groups possessing non-cyclic Sylow $p$-subgroups and a strongly $p$-embedded subgroup. The parameter $n$ is a positive integer and $r$ is a prime.}
\label{tbl:spe}
\end{table}

Proposition~\ref{prop:spe} has the following immediate corollary.

\begin{corollary}\label{propconsequences} Suppose that $G$ has a unique conjugacy class $\mathcal{C}$ of elements of order $p$. If $G$ does not have a strongly $p$-embedded subgroup then $\GG(\mathcal{C})$ is connected.
\end{corollary}

In our calculations, we are often in the situation where the normalizer $N_G(\langle t\rangle)$ is a maximal subgroup of $G$. In this situation we get the following result.

\begin{lemma}\label{lem:maxcent}
Suppose that $t\in G$ is conjugate to all of its non-trivial powers. If $N_G(\langle t\rangle)$ is maximal and $|\Lambda(t)|>|\langle t\rangle |-1$ then $\GG(t^G)$ is connected.
\end{lemma} 
\begin{proof}
Suppose that $u$ is adjacent to $t$ in $\GG(t^G)$ but that $u\not\in\langle t\rangle$. Then there is some $g\in G\setminus N_G(\langle t\rangle)$ such that $t^g=u$, and so since $t$ is conjugate to all of its non-trivial powers we deduce that $\langle N_G(\langle t\rangle),g\rangle\leq N_G(\Lambda(t))$ by Lemma~\ref{lem1} (here $\Lambda(t)$ is the component of $t$ in $\GG(t^G)$). But $N_G(\langle t\rangle)$ is maximal, whence $N_G(\Lambda(t))=G$. Therefore, $\GG(t^G)$ is connected by Lemma~\ref{lem:connectedstabilizer}.
\end{proof}

The next lemma is a mild generalization of this fact.

\begin{lemma}\label{lem:centralizermaximal} Let $t\in G$. Suppose that $t\in H\cap K$ for some subgroups $H$ and $K$ of $G$ with $G=\gen{H,K}$. If  $\GG(t^H)$ and $\GG(t^K)$ are connected then $\GG(t^G)$ is connected.
\end{lemma}
\begin{proof} Since $\GG(t^H)$ and $\GG(t^K)$ are subgraphs of $\GG(t^G)$, we see that the stabilizer of $\Lambda(t)$ contains $H$ and $K$. Thus $\GG(t^G)$ is connected by Lemma \ref{lem:connectedstabilizer}.

\end{proof}


For part of our work, for $G$ a quasisimple group and $t \in G\setminus Z(G)$, we would like to relate the graphs $\GG(t^G)$ and $\GG(\tilde t^{\widetilde G})$ where $\widetilde G=G/Z$ and $\tilde{t}=Zt$ for $Z \le Z(G)$.

\begin{lemma}\label{lem:pnotcenter} Let $p$ be a prime and let $Z$ be a central $p'$-subgroup of $G$. There is a bijection $\Phi$ between elements of order $p$ in $G$ and elements of order $p$ in $G/Z$, given by $g\mapsto Zg$. This bijection and its inverse preserve conjugacy, and if $\mathcal C$ is a union of conjugacy classes of elements of order $p$ in $G$ and $\widetilde{\mathcal C}$ is its image in $G/Z$, then $\Phi$ induces a graph isomorphism $\GG(\mathcal C)\to \GG(\widetilde{\mathcal C})$.
\end{lemma}
\begin{proof} That $\Phi$ is surjective is clear. To see injectivity, note that if $Zg=Zh$ then there is some $z\in Z$ such that $zg=h$, so $zg$ has order a power of $p$ and thus $z=1$. If $g$ and $h$ are conjugate in $G$ it is clear that $Zg$ and $Zh$ are conjugate in $\widetilde{G}=G/Z$. On the other hand if $Zg$ and $Zh$ are conjugate in $\widetilde{G}$ then $g^r=zh$ for some $r\in G$ and $z\in Z$, and again we deduce that $z$ must be trivial, whence $\Phi$ preserves conjugacy.

Let $\bar{\Phi}:\GG(\mathcal{C})\to\GG(\widetilde{\mathcal{C}})$ be the map induced by $\Phi$; we need to show that it preserves adjacency. Let $g,h\in \mathcal{C}$. Clearly, $[g,h]=1$ implies that $[Zg,Zh]=1$. Moreover, if $[Zg,Zh]=1$, then $[g,h] \in Z$. Hence $g^h=gz$ for some $z\in Z$ and, as $g^h$ has order $p$, we have $z=1$. Thus $[g,h]=1$. Since $\Phi$ preserves conjugacy, $gh^{-1}\in\mathcal{C}$ if and only if $ZgZh^{-1}\in\widetilde{\mathcal{C}}$, hence the result.
\end{proof}

Let $\Psi:G\mapsto G/Z(G)$ be the quotient map so that $\Phi$ is the restriction of $\Psi$ to the set of elements of order $p$ in $G$. If $G$ is a quasisimple group and $t\in G$ an element of order $p$ then Lemma~\ref{lem:componentgroup} shows that the component group of $t$ 
is generated by $t^{N_G(\Lambda(t))}=t^{\Psi^{-1}(N_{\widetilde{G}}(\Lambda(\tilde{t})))}$. Thus the component group for $G$ is just $O^{p'}(\Psi^{-1}(\Delta(\tilde{t})))$. This allows us to fairly easily compute the component group in any central extension $G$ of a simple group $G/Z(G)$ for primes not dividing $|Z(G)|$.

Thus we only need to understand the graphs for primes dividing $|Z(G)|$. For most types of quasisimple groups this means $p=2$ and sometimes $p=3$, with $p\geq 5$ only possible for special linear and special unitary groups.

The structure of the graphs at $p=2$ when $2\mid |Z(G)|$ can often be understood from the corresponding graph in the simple quotient. Given graphs $\Pi$ and $\Sigma$, the \emph{lexicographic product} $\Pi\cdot\Sigma$ is the graph with vertex set $V(\Pi)\times V(\Sigma)$ with $(u,x)$ adjacent to $(v,y)$ if and only if either $u$ is adjacent to $v$ in $\Pi$, or $u=v$ and $x$ and $y$ are adjacent in $\Sigma$.

\begin{lemma}\label{lem:coversingleclass} Suppose that $G$ is a group and $Z \le Z(G)$ is a $2$-group. Set $\widetilde{G}=G/Z$ and let $\varphi$ be the quotient map from $G$ to $\widetilde{G}$. Suppose that $\widetilde{\mathcal{C}}$ is a conjugacy class of non-central involutions in $\widetilde G$ such that $\mathcal{C}:=\varphi^{-1}(\widetilde{\mathcal{C}})$ is  a conjugacy class of involutions in $G$. Then $\GG(\mathcal{C})$ is the lexicographic product $\GG(\widetilde{\mathcal{C}})\cdot \overline{K_{|Z|}}$, where $\overline{K_{|Z|}}$ is the edgeless graph on $|Z|$ vertices. In particular, if $\GG(\widetilde{\mathcal{C}})$ is connected then $\GG(\mathcal{C})$ is connected.
\end{lemma}
\begin{proof} Let $x,y\in\widetilde{\mathcal{C}}$. To prove the result it suffices to show that whenever $u\in\varphi^{-1}(x)$ and $v\in\varphi^{-1}(y)$ then $u$ and $v$ are adjacent in $\GG(\mathcal{C})$ if and only if $x$ and $y$ are adjacent in $\GG(\widetilde{\mathcal{C}})$. This is straightforward, as $\varphi(uv)=xy$, and so $uv\in\varphi^{-1}(\widetilde{\mathcal{C}})=\mathcal{C}$ if and only if $xy\in\widetilde{\mathcal{C}}$. In addition since $uv\in\mathcal{C}$ and $xy\in \widetilde{\mathcal{C}}$, both are involutions and so $u$ and $v$ commute and $x$ and $y$ commute, hence the result.
\end{proof}

\medskip

\begin{remark} Lemma~\ref{lem:coversingleclass} does not extend to odd primes precisely because in the proof $x$ and $y$ might commute even though $u$ and $v$ do not.
\end{remark}

Given a group $G$ with conjugacy classes $\mathcal{C}_1,\mathcal{C}_2,\mathcal{C}_3$, the \emph{class multiplication coefficient} $n_G(\mathcal{C}_1,\mathcal{C}_2,\mathcal{C}_3)$ is the number of ways a fixed element $t\in \mathcal{C}_3$ can be written as a product $xy$ where $x\in\mathcal{C}_1$ and $y\in \mathcal{C}_2$. By \cite[Theorem 4.2.12]{Gorenstein}, 
\[n_G({\mathcal{C}_1,\mathcal{C}_2,\mathcal{C}_3})
=\frac{|\mathcal{C}_1||\mathcal{C}_2|}{|G|}\sum_{\chi\in\Irr(G)}\frac{\chi(g_1)\chi(g_2)\overline{\chi(g_3)}}{\chi(1)},\]
where $g_1$, $g_2$, and $g_3$ are representatives of the conjugacy classes $\mathcal{C}_1$, $\mathcal{C}_2$, and $\mathcal{C}_3$, respectively. In particular, when the character table of $G$ is known (as is the case for the sporadic groups, for example), the quantities $n_G(\mathcal{C}_1,\mathcal{C}_2,\mathcal{C}_3)$ can easily be computed.

The following lemma gives a very easy way to test the existence of edges in the graphs on involution classes. When $\mathcal{C}_1=\mathcal{C}_2=\mathcal{C}_3$, we write $n_G(\mathcal{C}_1):=n_G(\mathcal{C}_1,\mathcal{C}_1,\mathcal{C}_1)$.
\begin{lemma}\label{lem:edgeexist}
Let $G$ be a group and $\mathcal{C}$ be a conjugacy class of involutions in $G$. Then $\GG(\mathcal{C})$ has an edge if and only if $n_G(\mathcal{C})\ne 0$. In particular, if, for $t \in \mathcal C$, $C_G(t)$ is a maximal subgroup of $G$ and $n_G(\mathcal{C})\ne 0$, then $\GG(\mathcal{C})$ is connected.
\end{lemma}
\begin{proof}
If $n_G(\mathcal{C})>0$ then there exist $x,y\in\mathcal{C}$ such that $xy\in\mathcal{C}$. Moreover, since $\mathcal{C}$ is a class of involutions, $xy\in\mathcal{C}$ implies that $x$ and $y$ commute. Thus, the adjacency conditions for $\GG(\mathcal{C})$ are satisfied by $x$ and $y$. The converse is clear. The connectivity statement follows from Lemma~\ref{lem:maxcent}.
\end{proof}
\begin{lemma}\label{lem:cliques}
Let $G$ be a group and $\mathcal{C}=t^G$ be a conjugacy class of elements of order $p$ with $\mathcal{C}=\overline{\mathcal{C}}$. If $n_G(\mathcal{C})=p-2$ then $\Lambda(t)$ is the complete graph $K_{p-1}$.
\end{lemma}
\begin{proof} By Lemma~\ref{lem:rat1}, the subgraph of $\GG(t^G)$ induced on the powers of $t$ is $K_{p-1}$. Each neighbour $u$ of $t$ corresponds to the pair $(u,u^{-1}t)$ with product $t$. Since $t$ has $p-2$ neighbours which are powers of $t$, it follows that if $t$ is adjacent to vertices in $G\setminus \gen{t}$ then $n_G(\mathcal{C})>p-2$.
\end{proof}
\begin{lemma}\label{lem:covermulticlass}

Suppose that $G$ is a group and $Z \le Z(G)$ is a $2$-group. Set $\widetilde{G}=G/Z$ and let $\varphi$ be the quotient map from $G$ to $\widetilde{G}$. Suppose that $\widetilde{\mathcal{C}}$ is a conjugacy class of non-central involutions in $\widetilde G$ such that $ \varphi^{-1}(\widetilde{\mathcal{C}})= \bigsqcup_{j=1}^m\mathcal{C}_j $ is a union of $m>1$ conjugacy classes of involutions in $G$. Assume that $\GG(\widetilde{\mathcal{C}})$ is connected. If $n_G(\mathcal{C}_1,\mathcal{C}_1,\mathcal{C}_j)=0$ for all $1< j\leq m$, then $\GG(\mathcal{C}_1)$ is connected and $\GG(\mathcal{C}_j)$ is edgeless for all $1<j\leq m$.
\end{lemma}
\begin{proof} Suppose that $\{Zx,Zy\}$ is an edge in $\GG(\widetilde{\mathcal C})$. Let $u\in \mathcal C_1 \cap Zx$ and $v \in \mathcal C_1 \cap Zy$. Then $Z xy\in \widetilde {\mathcal{C}}$, hence $uv \in \bigsqcup_{j=1}^m\mathcal{C}_j$. As $n_G(\mathcal{C}_1,\mathcal{C}_1,\mathcal{C}_j)=0$ for all $1<j\le m$, we have $uv \in \mathcal C_1$. In particular, $uv$ is an involution and consequently $u$ and $v$ commute. Hence $u$ and $v$ are adjacent in $\GG(\mathcal{C}_1)$. It follows that paths in $\GG(\widetilde{\mathcal{C}})$ lift to paths in $\GG(\mathcal{C}_1)$, and so
$\GG(\mathcal {C}_1)$ is connected.


For the second claim suppose for contradiction that $j>1$ and that there exist $w,r,s\in \mathcal{C}_j$ satisfying $rs=w$. Take $z\in Z$ such that $rz\in\mathcal{C}_1$. Then $1=(rz)^2=r^2z^2=z^2$ and so $z$ is an involution. As $r$ and $s$ are conjugate and $z$ is central, we have $sz\in \mathcal C_1$. Now  $(rz)(sz)=rsz^2=rs=w$, contradicting the assumption that $n_G(\mathcal{C}_1,\mathcal{C}_1,\mathcal{C}_j)=0$, hence the result.
\end{proof}




\section{Sporadic simple groups}\label{s:simple}
Throughout the remaining sections we shall frequently rely on data from the $\mathbb{ATLAS}$~\cite{atlas}. Information regarding characters, maximal subgroups, and normalizers in quasisimple groups, whenever asserted without citation, may be found in the $\mathbb{ATLAS}$. We also use the $\mathbb{ATLAS}$ names for conjugacy classes of simple groups, and for covers and subgroups of the simple groups we shall use names agreeing with the {\sf GAP} Character Table Library~\cite{CTblLib1.3.11}. Our notation for abstract groups will also be based on the $\mathbb{ATLAS}$: $N.H$ denotes an extension of a normal subgroup $N$ by a group $H$, and when important to indicate whether or not the extension splits, we shall use $N\cn H$ and $N\ns H$ for split and non-split extensions, respectively. For a prime $p$ the elementary abelian $p$-group of order $p^a$ is denoted by $p^a$, the notation $p^{1+2a}_+$ when $p$ is odd denotes an extraspecial $p$-group of order $p^{1+2a}$ and exponent $p$ while $2^{1+4}_-$ is the central product of $Q_8$ and $D_8$, and $2^{1+6}_+$ is the central product of three copies of $D_8$. Finally, $[p^a]$ is a $p$-group of order $p^a$ whose structure is not made explicit.

The computer program \texttt{FindStabilizerOfComponent} in the supplementary materials~\cite{codesup} can be used very quickly on `small' sporadic quasisimple groups, specifically those with a permutation representation on fewer than roughly 500 000 points. The sporadic groups $\HN$ and $\Fi_{24}'$ take longer, but are still possible to complete with the general program. We include a function \texttt{TestSporadicGroup} that takes as an argument a string like \texttt{m11} and checks that the graphs for that group have the properties stated in this article. We also include a function \texttt{TestMatrixSporadicGroup}, which checks the specific claims made about the groups $\Ly$, $\Th$, and $\J_4$.

The general structure of the graphs $\GG(t^G)$, component groups $\Delta(t)$, and component stabilizers $N_G(\Lambda(t))$ for sporadic simple groups is given in the following three results. We may assume that $p^2\mid |G|$, as otherwise the graph is disconnected and the component group is cyclic of order $p$. Throughout the rest of the paper, whenever $\Lambda(t)\subseteq\langle t\rangle$ we shall omit an explicit description of the stabilizer $N_G(\Lambda(t))$ as this group is simply either $C_G(t)$ or $N_G(\langle t\rangle)$, and the description of such subgroups is readily found in other sources.

\begin{proposition}\label{prop:explicit1} Let $G$ be a sporadic simple group and $t$ be an element of prime order $p$. Suppose that a Sylow $p$-subgroup does not have order $p$. Then $\Lambda(t)\subseteq \langle t\rangle$ if and only if either $G=\HN$, $t^G\in\{\mathrm{5C,5D}\}$, and $\Lambda(t)=\{t\}$, or else $|\Lambda(t)|=p-1$ and $(G,t^G)$ appear in Table~\ref{tbl:prop3.1}. In the former case, $\GG(\overline{t^G})$ is connected, and in all other cases $\Lambda(t)=\langle t\rangle\setminus\{1\}$.
\end{proposition}
\begin{table}[H]
\centering
    \begin{tabular}{|Ll|Ll|Ll|Ll|}\hline
        G & $t^G$ &G & $t^G$ &G & $t^G$ &G & $t^G$ \\\hline
         \HS & 5A& \Suz & 3A & \Fi_{22} &2A & \Ly &3A\\
         \J_2&3A&\Co_3&5A&\Fi_{23}&2A&\J_4&11A\\
         \McL&3A, 5A&\Co_2&2A, 3A, 5A&\Co_1&3A&\BM&2A\\\hline
    \end{tabular}
    \caption{Sporadic simple groups and classes $t^G$ satisfying the conditions of Proposition~\ref{prop:explicit1} with $\Lambda(t)=\langle t\rangle\setminus\{1\}$.}
\label{tbl:prop3.1}
\end{table}
\begin{proposition}\label{prop:explicit2} Let $G$ be a sporadic simple group and $t$ be an element of prime order $p$. Suppose that $\Delta(t)\cong C_p\times C_p$. Then the pairs $(G,t^G)$ appear in Table~\ref{tbl:prop3.2}. Moreover, for the first column, $|\Lambda(t)|=p(p-1)$, and for the second column $|\Lambda(t)|=p^2-1$. For $\J_2$ and $\He$ the classes are not rational, and there are 6 and 21 elements in $\Lambda(t)$ respectively, with twice as many in the component in $\GG(\overline{t^G})$.
\end{proposition}
\begin{table}[H]
\begin{center}
\begin{tabular}{|LlL|LlL|LlL|}\hline
G    & $t^G$ & N_G(\Lambda(t))                          & G            & $t^G$ & N_G(\Lambda(t))  & G    & $t^G$  & N_G(\Lambda(t)) \\\hline
\HS  & 5C    & [5^3].[2^2]                           & \M_{11}      & 3A    & 3^2.[2^4]     & \J_2 & 5C, 5D & 5^2.3.[2^2]  \\
\McL & 5B    & [5^3].[2^2]                           & \twist\F_4(2)' & 5A    & 5^2.[2^3].3.2 & \He  & 7D, 7E & [7^3].3.2     \\
\ON  & 7B    & [7^3].3.2                             & \J_3         & 3B    & [3^5].[2^3]   &      &        &              \\
\J_4 & 11B   &  11^{1+2}_+.10 &              &       &               &      &        &           \\ \hline
\end{tabular}
\end{center}
\caption{Sporadic simple groups and classes $t^G$ satisfying the conditions of Proposition~\ref{prop:explicit2}.}
\label{tbl:prop3.2}
\end{table}
\begin{proposition}\label{prop:explicit3} Let $G$ be a sporadic simple group and let $t$ be an element of prime order $p$. Suppose that $|\Delta(t)|>p^2$ but that $\GG(t^G)$ is disconnected. Then $(G,t^G)$ are as in Table~\ref{tbl:prop3.3}.
\end{proposition}

\begin{table}[h]
\begin{center}
\begin{tabular}{|lllll|lllll|}\hline
$G$     & $t^G$ & $|\Lambda(t)|$ & $\Delta(t)$    & $N_G(\Lambda(t))$ & $G$      & $t^G$ & $|\Lambda(t)|$ & $\Delta(t)$            & $N_G(\Lambda(t))$                      \\\hline
$\Suz$  & 5B    & 4032        & $\J_2$      & $\J_2.2$       & $\Co_2$  & 5B    & 147840      & $\HS$               & $\HS.2$                             \\
$\HN$   & 5E    & 400         & $5^{1+4}_+$ &  $5^{1+4}_+.2^{1+4}_-.5.2$            & $\Fi_{22}$ & 5A    & 1741824     & $\mathrm{O}_8^+(2)$ & $\mathrm{O}_8^+(2).\mathrm{Sym}(3)$ \\
$\Co_3$ & 5B    & 147840      & $\HS$       & $\HS$          & $\Ly$     & 5B    & 2400        & $5^{1+4}_+$         & $5^{1+4}_+.4.\Sym(6)$              \\\hline
\end{tabular}%
\end{center}
\caption{Sporadic simple groups and classes $t^G$ satisfying the conditions of Proposition~\ref{prop:explicit3}.}
\label{tbl:prop3.3}
\end{table}

We now prove each of these propositions, which will collectively prove Theorem~\ref{thm:mainresults}(a). Given a conjugacy class $\mathcal{C}$ of $G$, we say that a subgroup $H<G$ is $\mathcal{C}$-\emph{pure} if $\mathcal{C}\cap H=H\setminus\{1\}$.

\begin{proof} We use the program \texttt{TestSporadicGroup} in the supplementary materials, which proves the propositions above for all simple groups apart from $\Ly$, $\Th$, $\J_4$, $\BM$, and $\MM$. We must therefore deal with each of those groups in turn. The claims about $\Ly$, $\Th$, and $\J_4$ that we make here are verified using the {\sc Magma}~\cite{magma} program \texttt{TestMatrixSporadicGroup} in the supplementary materials. We remark that Proposition~\ref{prop:explicit1} can be proved directly for the involution classes from character tables by appealing to Lemma~\ref{lem:cliques}.

If $p=2$ and $C_G(t)$ is maximal, then either $\GG(t^G)$ is connected or it is edgeless by Lemma~\ref{lem:maxcent}. Moreover, we may assume that $G$ has at least two conjugacy classes of involutions by Proposition~\ref{prop:spe} and Corollary~\ref{propconsequences}.

Thus we can use class multiplication coefficients to determine the connectivity of $\GG(t^G)$. The remaining pairs $(G,t^G)$ for which $t$ is an involution with $C_G(t)$ maximal are displayed in Table~\ref{tbl:maxcent}, together with the value $n_G(t^G)$, as computed using the {\sf GAP} Character Table Library~\cite{CTblLib1.3.11}.
\begin{table}[ht]
    \begin{center}
    \begin{tabular}{|lll|lll|}\hline
    $G$ & $t^G$ & $n_G(t^G)$&$G$ & $t^G$ & $n_G(t^G)$\\
    \hline
    $\J_4$ & $\mathrm{2A}$ & 112266&$\MM$   & $\mathrm{2A}$ & 27143910000\\
     $\BM$   & $\mathrm{2A}$ & 0&&$\mathrm{2B}$ & 90717803016750\\
       & $\mathrm{2B}$ & 7379550&&&\\
                   & $\mathrm{2C}$ & 184246272&&& \\ \hline \end{tabular}\end{center}
\caption{Some involution classes and their class multiplication coefficients.}
\label{tbl:maxcent}
\end{table}
Consequently, we deduce from Lemma~\ref{lem:edgeexist} that $\GG(t^G)$ is connected for each such pair $(G,t^G)$ aside from $(\BM,\mathrm{2A})$ and in this case $\GG(t^G)$ is edgeless. The only involution classes remaining to consider are $\mathrm{2B}$ in $\J_4$ and $\mathrm{2D}$ in $\BM$.

\medskip
\textbf{Lyons group}: For $p=3$ there are two classes: 3A has centralizer $3\ns\McL$ and 3B has a centralizer of order $174960$. Suppose $t\in\mathrm{3A}$; the group $3\ns\McL$ has six classes of elements of order $3$, with representatives $t$, $t^2$, $a$, $at$, $at^2$, and $b$, where $t$ is central and $a$ and $b$ are two elements of order $3$. By a computer calculation, $t$, $t^2$, and $a$ lie in class 3A in $G$ and the others lie in class 3B. Moreover, the class of $a$ is real in $3\ns\McL$. Therefore we see that in $\Lambda(t)$ the only edge incident to $t$ is from $t$ to $t^2$, as claimed.

Now we suppose that $t$ belongs to class 3B. We work inside the maximal subgroup $L=3^{2+4}.2.\Alt(5).D_8$, which contains a 3B-element and its centralizer. Inside here we find that $L$ is contained in the stabilizer $N_G(\Lambda(t))$, and that $L$ does not contain the full centralizer of every element in $\Lambda(t)\cap L$. Thus $L<N_G(\Lambda(t))$, and consequently $N_G(\Lambda(t))=G$, so the claim follows from Lemma~\ref{lem:connectedstabilizer}. (See the supplementary materials for a proof of this.)

For $p=5$ there are two classes, with 5A having centralizer $5^{1+4}_+.\Sym(6)$ (and maximal normalizer) and 5B having centralizer of order $3750=2\cdot 3\cdot 5^4$. The maximal subgroup $5^3.\mathrm{SL}_3(5)$ contains a Sylow $5$-subgroup $S$ of $G$ and has a normal elementary abelian subgroup $Q$ of order $5^3$. Since $Q$ is normal in $S$ it contains an element of $Z(S)$ and hence of conjugacy class 5A. Moreover, $\mathrm{SL}_3(5)$ is transitive on non-zero vectors, and so $Q$ is a 5A-pure subgroup. Therefore $\GG(t^G)$ is connected for $t\in\mathrm{5A}$ by Lemma~\ref{lem:maxcent}.

For $t\in\mathrm{5B}$, we work entirely inside the maximal subgroup $H=5^{1+4}_+.4.\Sym(6)$. A computer calculation in the supplementary materials shows that $H$ contains the centralizer of a 5B element, so we can find all neighbours of $t$ in $\GG(t^G)$ inside $H$. They are all $H$-conjugate to $t$, so $N_G(\Lambda(t))$ lies inside $H$ as well by Lemma~\ref{lem1}. In fact, one can check that $H$ \emph{is} the stabilizer of $\Lambda(t)$, and that $\Delta(t)$ is the normal subgroup $5^{1+4}_+$ of $H$. This is performed in the supplementary materials.

\medskip

\textbf{Thompson group}: There is a single conjugacy class each of elements of orders $5$ and $7$, and three classes of elements of order $3$. When $p>3$, our results follow from Corollary~\ref{propconsequences} and Table~\ref{tbl:spe}. Thus we can focus on the case $p=3$.

Here we look in the subgroup $H=(3\times \mathrm G_2(3)).3$. As confirmed by computation performed in the supplementary materials, $H$ has 41 classes of elementary abelian subgroups of order $3^2$; we find such subgroups $P$, $Q$, and $R$, where at least six elements of $P$ lie in 3A, at least six elements of $Q$ lie in 3B, and at least six elements of $R$ lie in 3C. Thus $\GG(t^G)$ has an edge connecting $t$ with some element not in $\gen t$ for $t$ in any of the classes of elements of order $3$. Since $N_G(\gen t)$ is maximal in all three cases, we see that the stabilizer of $\Lambda(t)$ is $G$, and hence $\GG(t^G)$ is connected for each $t$ of order $3$ by Lemma~\ref{lem:connectedstabilizer}.

\medskip

\textbf{Janko group $\J_4$}: There is a single conjugacy class of elements of order $3$, and two classes each for orders $2$ and $11$. The normalizers of $\gen t$ for $t$ in classes 2A and 11A are maximal, but not for the other classes. For $t\in\mathrm{3A}$, we may apply Corollary~\ref{propconsequences} and Table~\ref{tbl:spe} to see that $\GG(t^G)$ is connected.

For class 2B we work inside the maximal subgroup $L=2^{11}\cn\M_{24}$ (importantly a split extension) and set $Q=O_2(L)$. Here we let $t$ be one of the 276 elements of $Q$ that are in conjugacy class 2B. There are 2B-pure Klein four subgroups in $Q$, so let $u$ be such that $\gen{t,u} \leq Q $ is 2B-pure. A computer calculation in the supplementary materials shows that $L=\gen{C_L(t),C_L(u)}$, and so $L\leq N_G(\Lambda(t))$. On the other hand, in the supplementary materials we find that there is a 2B-pure Klein four subgroup $\gen{t,v}$ of $L$ that is not in $Q$. Thus $C_G(v)\not\leq L$, and hence $\Lambda(t)$ is stabilized by $\gen{L,C_G(v)}=G$, and the claim follows from Lemma~\ref{lem:connectedstabilizer}.


Finally, consider $p=11$. The Sylow $11$-subgroups of $G$ have trivial intersection, and so the connected components are easy to determine. The centralizer of 11A has order 132 times that of 11B. As the Sylow $11$-subgroup is $P\cong 11^{1+2}_+$, we see that $Z(P)\setminus\{1\}$ has type 11A and $P\setminus Z(P)$ consists of 11B elements. In particular, both classes are rational. Hence $\Lambda(t)$ consists of the non-identity powers of $t$ for $t$ in 11A, and has order $11(11-1)$ for $t$ in 11B. In particular, when $t\in \mathrm{11B}$, $\Delta(t)$ has order $11^2$ and $N_G(\Lambda(t))=N_G(\Delta(t)) = N_G(\gen{t})P$ and so has order $2\cdot 5\cdot 11^3$. We obtain the result in Proposition~\ref{prop:explicit2}.
\medskip

\textbf{Baby Monster $\BM$}: Here the relevant classes are 2A, 2B, 2C, 2D, 3A, 3B, 5A, 5B, 7A. In all cases except for 2D and 7A the normalizer of $\gen t$ is maximal in $G$ (and the conjugacy class is a rational class), so all we have to do is find some element in $\Lambda(t)$ that lies outside $\gen t$, and then $\GG(t^G)$ is connected by Lemma~\ref{lem:maxcent}. Since 7A is the unique class of elements of order 7 and $\BM$ has no strongly 7-embedded subgroup (see Table~\ref{tbl:spe}), it follows from Corollary~\ref{propconsequences} that $\GG(t^G)$ is connected for $t$ in 7A.

When $p=2$, the classes $\mathrm{2A},\mathrm{2B}$, and $\mathrm{2C}$ are handled in Table~\ref{tbl:maxcent}.
It remains to consider $t\in\mathrm{2D}$. The involutions from the Thompson group $\Th$, a subgroup of $G$, lie in class 2D of $G$, as a simple character value calculation shows (the character of degree $4371$ restricts to $\Th$ as $248+4123$).

Thus, choosing $K_1$ to be a subgroup $\Th$ of $G$ containing $t$, we have that $\GG(t^{K_1})$ is connected. Therefore the maximal subgroup $K_1$ is a subgroup of $N_G(\Lambda(t))$, but so is $C_G(t)$, which is not contained in $K_1$. Thus $\GG(t^G)$ is connected for class 2D as well by Lemma~\ref{lem:centralizermaximal}.

For $p=3$, we again choose a subgroup $K_2=\Th$ containing $t$, as $\GG(t^{K_2})$ is connected for all elements of order $3$. The three classes 3A, 3B, and 3C in $\Th$ have traces $78$, $-3$, and $-3$ on $248+4123=4371$. By looking at the traces of 3A ($78$) and 3B ($-3$) in $\BM$, we see that 3A fuses into 3A and 3B fuses into 3B. Thus the connected component of $t$ in $G$ contains elements outside $\gen t$, since the same is true in $K_2$.

Finally, we let $p=5$. Note that class 5A has centralizer $5 \times \HS\cn2$, and has Sylow $5$-subgroups of order $5^4$, whereas 5B has centralizer $5^{1+4}_+.2^{1+4}_-.\Alt(5)$, and is central in a Sylow $5$-subgroup of $G$. There is a maximal subgroup $(5^2\cn 4.\Sym(4)) \times \Sym(5)$ in $G$. The $5^2$ here must be 5A-pure, because each element centralizes a $\Sym(5)$ whereas the 5B centralizer has no such subgroup. On the other hand, there is a maximal subgroup $5^3.\mathrm{L}_3(5)$ which contains a Sylow $5$-subgroup $S$. Since the normal $5^3$ subgroup contains $Z(S)$, and $\mathrm{L}_3(5)$ acts transitively on the $5$-elements in the $5^3$, the $5^3$ is 5B-pure. Thus $\GG(t^G)$ is connected for both classes of elements of order $5$.

\medskip

\textbf{Monster $\MM$}: Here the relevant classes are 2A, 2B, 3A, 3B, 3C, 5A, 5B, 7A, 7B, 11A, 13A, 13B. In all cases except for 11A the normalizer of $\gen t$ is maximal in $G$ and the relevant conjugacy class is a rational class, so we need that there is an element of $\Lambda(t)$ not in $\gen t$. The involution classes have been dealt with in Table~\ref{tbl:maxcent}. The Monster has a unique class of elements of order 11 but has no strongly 11-embedded subgroup. Therefore, $\GG(t^G)$ is connected for $t\in \mathrm{11A}$ by Corollary~\ref{propconsequences}.

The embedding of $2\ns\BM$ into $\MM$ restricts the $196883$ character of $\MM$ to a sum of characters $96256+96255+4371+1$ for $2\ns\BM$. Thus we may check that the labels for the classes 3A, 3B, 5A, 5B and 7A are all the same for $K=2\ns\BM$ and $G$. By the argument above for the group $\BM$, together with Lemma~\ref{lem:pnotcenter}, we deduce that $\GG(t^K)$ is connected for each of the relevant classes in $K$. Therefore, $\GG(t^G)$ is connected as well.

For the class 3C, the normalizer is $S_3\times\Th$, so it suffices to find an element of 3C not in the direct factor $S_3$, as then the component stabilizer will contain two distinct maximal subgroups $S_3\times\Th$ and thus the whole group. This is easy to verify by checking class fusion using the {\sf GAP} Character Table Library~\cite{CTblLib1.3.11}. The same proof works for 13A, where the normalizer is $(13\cn6 \times \mathrm{L}_3(3))\cn2$, and it suffices to find an element of 13A not in the direct factor $13\cn6$. We again check this using the GAP Character Table Library (classes 13D, 13E, 13F, and 13G from the maximal subgroup fuse to 13A of $\MM$).

We thus would like a 7B-pure subgroup of order $49$, and a 13B-pure subgroup of order $169$ to complete the proof. While it is possible to find edges in these graphs without just checking maximal subgroups, the easiest way is to note that the $7$-core of the maximal subgroup $7^2.\mathrm{SL}_2(7)$ and the $13$-core of $13^2.\mathrm{SL}_2(13).4$ have the correct requirements \cite[Sections 2.20 and 2.24]{dlpp2024un}. Thus the Monster only appears silently in Propositions~\ref{prop:explicit1}, \ref{prop:explicit2}, and \ref{prop:explicit3}.
\end{proof}

\medskip
Together, Lemma~\ref{lem:notcyclic}, Propositions~\ref{prop:explicit1}, \ref{prop:explicit2}, and \ref{prop:explicit3} prove Theorem~\ref{thm:mainresults}(a).

\section {Central extensions}\label{s:covers}
We conclude the paper by considering the case $Z(G)\ne 1$ in Theorem~\ref{thm:mainresults}.

\subsection{Covers of sporadic simple groups}

Fix the notation as in Lemma~\ref{lem:coversingleclass}. In particular, $\widetilde{\mathcal{C}}$ is the image of the $G$-conjugacy class $\mathcal{C}$ in $G/Z(G)$. Note that $|Z(G)|\le 4$.

\begin{lemma}\label{cor:coverconn}
Suppose that $(G,\widetilde{\mathcal{C}})$ is one of $(4\ns\M_{22},\mathrm{2A})$, $(2\ns\Fi_{22},\mathrm{2C})$, $(2\ns\Co_1,\mathrm{2C})$, or $(2\ns\BM,\mathrm{2D})$. Then $\GG(\mathcal{C})$ is connected.
\end{lemma}
\begin{proof}
The result for $4\ns\M_{22}$ is verified in the supplementary materials. In each remaining case $\mathcal{C}$ is the unique $G$-conjugacy class with image $\widetilde{\mathcal{C}}$, so Lemma~\ref{lem:coversingleclass} applies. The result now follows since $\GG(\widetilde{\mathcal{C}})$ is connected by Theorem~\ref{thm:mainresults}(a) and Propositions~\ref{prop:explicit1},~\ref{prop:explicit2}, and~\ref{prop:explicit3}.
\end{proof}
The rest of the cases where $Z(G)$ is a $2$-group can be handled by examining the class multiplication coefficients. We use the {\sf GAP} class labels for the classes of $G$.
\begin{theorem}
Let $G$ be a 2-cover of a sporadic simple group, and $\mathcal{C}$ be an involution class. Either $\GG(\mathcal{C})$ is connected, or else $\GG(\mathcal{C})$ is edgeless and \[(G,\mathcal{C})\in S:=\{(2\ns\J_2,\mathrm{2c}),(2\ns\Fi_{22},\mathrm{2b}),(2\ns\Fi_{22},\mathrm{2c}),(2\ns\BM,\mathrm{2b})\}.\]
\end{theorem}
\begin{proof}
We first compute in {\sf GAP} that $n_G(\mathcal{C})=0$ for all $(G,\mathcal{C})\in S$, and so such graphs are indeed edgeless by Lemma~\ref{lem:edgeexist}. Additionally, that the graphs on 2A-lifts in $2\ns\M_{22}$ are connected is verified in the supplementary materials. Note now that all remaining pairs $(G,\mathcal{C})$ which are neither in $S$ nor covered by Lemma~\ref{cor:coverconn} are such that the centralizer in $G$ of an element of $\mathcal{C}$ is maximal~\cite{atlas}. Moreover, in each such case $n_G(\mathcal{C})\ne0$, as displayed in Table~\ref{tab:doublecovercoefficients}. Hence the result now follows from Lemma~\ref{lem:edgeexist}.
\end{proof}

\begin{table}[H]
   \begin{center}
\begin{tabular}{|llll|llll|}
\hline
$G$           & $\mathcal{C}$ & $\widetilde{\mathcal{C}}$ & $n_G(\mathcal{C})$ & $G$            & $\mathcal{C}$ & $\widetilde{\mathcal{C}}$ & $n_G(\mathcal{C})$ \\ \hline
$2\ns\M_{12}$ & $\mathrm{2b}$ & $\mathrm{2B}$             & $24$               & $2\ns\Suz$     & $\mathrm{2b}$ & $\mathrm{2A}$             & $54$               \\
              & $\mathrm{2c}$ & $\mathrm{2B}$             & $6$                &                & $\mathrm{2c}$ & $\mathrm{2A}$             & $360$              \\\hline
$2\ns\J_2$    & $\mathrm{2b}$ & $\mathrm{2A}$             & $10$               & $2\ns\Fi_{22}$ & $\mathrm{2d}$ & $\mathrm{2B}$             & $1512$             \\
              &               &                           &                    &                & $\mathrm{2e}$ & $\mathrm{2B}$             & $270$              \\\hline
$2\ns\HS$     & $\mathrm{2b}$ & $\mathrm{2A}$             & $80$               & $2\ns\Co_1$    & $\mathrm{2b}$ & $\mathrm{2A}$             & $12600$            \\
              & $\mathrm{2c}$ & $\mathrm{2A}$             & $30$               &                & $\mathrm{2c}$ & $\mathrm{2A}$             & $270$              \\\hline
$2\ns\Ru$     & $\mathrm{2b}$ & $\mathrm{2A}$             & $1120$             & $2\ns\BM$      & $\mathrm{2c}$ & $\mathrm{2B}$             & $7286400$          \\
              & $\mathrm{2c}$ & $\mathrm{2A}$             & $270$              &                & $\mathrm{2d}$ & $\mathrm{2B}$             & $93150$           \\\hline
\end{tabular}
\end{center}
\caption{The double covers of sporadic groups and the involution classes with maximal centralizer and non-zero class multiplication coefficient. We use $\mathcal{C}$ to denote the {\sf GAP} class labelling, and $\widetilde{\mathcal{C}}$ is the image in the simple quotient $G/Z(G)$.}
\label{tab:doublecovercoefficients}
\end{table}
We conclude this subsection by presenting Table~\ref{tbl:3covers}, which summarizes the relevant data for the triple covers of sporadic simple groups---we emphasize that any relevant class of elements of order 3 not appearing in this table has a connected graph. Entries are blank in the final column whenever $\Lambda(t)\subset \langle t\rangle$, as in this case $N_G(\Lambda(t))=N_G(\langle t\rangle)$, and the structure is well-known. The claims are proved computationally in the supplementary materials. Especially for use when examining $3\ns\Fi_{24}'$, we generate $N_G(\Lambda(t))$ using Lemma~\ref{lem1variant}. The code in the supplementary materials uses the fact that elementary abelian subgroups $E$ of order $9$ can be distinguished by the fixed-point multisets $\{\Fix(t): t \in E\}$ and so is not as transportable to other Gill--Guillot graph calculations as \texttt{FindStabilizerOfComponent}. This concludes the proof of Theorem~\ref{thm:mainresults}(b).

\begin{table}[h]
\centering
\begin{tabular}{|lllcllc|}
\hline
$G$             & image         & real? & no. classes & $\Delta(t)$ & $|\Lambda(t)|$ & $N_G(\Lambda(t))$ \\ \hline
$3\ns\M_{22}$   & $3\mathrm{A}$ & yes   & 1           & $3^{2}$     & 6              & $3^2.2^2.3.2$     \\ \hline
$3\ns\J_{3}$    & $3\mathrm{A}$ & yes   & 1           & $3^{1}$     & 2              & --                \\
                & $3\mathrm{B}$ & yes   & 1           & $3^{3}$     & 24             & $[3^6].[2^3]$     \\ \hline
$3\ns\McL$      & $3\mathrm{A}$ & yes   & 1           & $3^{1}$     & 2              & --                \\
                & $3\mathrm{A}$ & no    & 2           & $3^{1}$     & 1              & --                \\
                & $3\mathrm{B}$ & yes   & 1           & $3^{5}$     & 180            & $3^5.\M_{10}$   \\ \hline
$3\ns\Suz$      & $3\mathrm{A}$ & yes   & 1           & $3^{1}$     & $2$            & --                \\
                & $3\mathrm{A}$ & no    & 2           & $3^{1}$     & $1$            & --                \\ \hline
$3\ns\ON$       & --            & --    & --          & --          & --             & --                \\ \hline
$3\ns\Fi_{22}$  & $3\mathrm{A}$ & no    & 2           & $3^{1}$     & $1$            & --                \\ \hline
$3\ns\Fi_{24}'$ & $3\mathrm{B}$ & no    & 2           & $3^1$       & $1$            & --                \\
                & $3\mathrm{C}$ & no    & 2           & $3^1$       & $1$            & --                \\ \hline
\end{tabular}
\caption{Classes for which $\GG(\mathcal{C})$ is not connected. The entry `real?' is used as needed to distinguish between classes with the same image in $G/Z(G)$. When the entry `no. classes' is 2, there are \emph{two} graphs, both with the same component stabilizer.} \label{tbl:3covers}
\end{table}

\subsection{Exceptional covers of alternating groups and groups of Lie type} Finally, we consider the quasisimple groups which are exceptional covers of non-sporadic simple groups.

Table~\ref{tbl:exco} displays the classes which yield disconnected graphs, their images in the corresponding simple quotient, the size of the connected components of their Gill--Guillot graphs, and their component groups; classes in the covers are labelled as in the {\sf GAP} Character Table Library~\cite{CTblLib1.3.11}, and their images are labelled using the $\mathbb{ATLAS}$~\cite{atlas} naming.

We now prove the validity of the information in Table~\ref{tbl:exco}, from which Theorem~\ref{thm:mainresults}(c) follows.

\begin{proof}[Proof of validity of Table~\ref{tbl:exco}]
If $G\notin\{2^2\ns\mathrm{O}_8^+(2),2\ns\F_4(2), 2\ns\twist\E_6(2),2^2\ns\twist\E_6(2)\}$ then the claims are verified in the supplementary materials. We handle the remaining groups separately; throughout, we write $\widetilde{G}:=G/Z(G)$.

Suppose $G=2^2\ns\mathrm{O}_8^+(2)$. Then each involution class of $\widetilde{G}$ has connected Gill--Guillot graph by~\cite[Theorem 2]{GGL25} and Lemma~\ref{lem:connectedstabilizer}. Moreover, $G$ has exactly two classes of non-central involutions, with distinct images in $\widetilde{G}$~\cite{atlas}. Consequently, both classes have connected graphs by Lemma~\ref{lem:coversingleclass}, as claimed.

We now suppose $G=2\ns\F_4(2)$. This group $G$ has eight conjugacy classes of non-central involutions, two corresponding to each such class in the simple quotient $\widetilde{G}$. The classes 2f, 2g, 2h, and 2i (with images 2C and 2D) are handled in the supplementary materials. For the remaining four classes we shall use Lemma~\ref{lem:covermulticlass}. The Gill--Guillot graphs on involution classes in $\widetilde{G}$ are all connected by~\cite[Theorem 2]{GGL25}. We compute using the {\sf GAP} Character Table Library~\cite{CTblLib1.3.11} that \[n_G(\mathrm{2b,2b,2c})=n_G(\mathrm{2d,2d,2e})=0.\]
In particular, the claims for $G$-conjugacy classes 2b, 2c, 2d, and 2e follow from Lemma~\ref{lem:covermulticlass}.

Finally, consider the case $\widetilde{G}=\twist\E_6(2)$. Then $G$ has exactly $|Z(G)|$ involution classes with image $\mathrm{2A}$, $|Z(G)|$ involution classes with image $\mathrm{2B}$, and exactly one involution class with image $\mathrm{2C}$~\cite{atlas}. Each of the involution classes in $\widetilde{G}$ gives a connected graph by~\cite[Theorem 2]{GGL25}, so the $\mathrm{2C}$-lift must also have a connected graph in $G$ by Lemma~\ref{lem:coversingleclass}. Consider now the $\mathrm{2A}$-lifts; for simplicity of notation here we shall assume that $|Z(G)|=4$; the other case is nearly identical. We compute using the {\sf GAP} Character Table Library that $n_G(\mathrm{2d},\mathrm{2d},\mathcal{C})=0$ for all $\mathcal{C}\in\{\mathrm{2e,2f,2g}\}$, hence the claimed result follows from Lemma~\ref{lem:covermulticlass}.

For the $\mathrm{2B}$-lifts we shall consider the two covers separately, starting with the case $|Z(G)|=2$. We consider the maximal subgroup $F:=2\ns\F_4(2)$: we check using the {\sf GAP} Character Table Library that the $F$-conjugacy classes labelled $\mathrm{2d}$ and $\mathrm{2f}$ fuse to the 2B-lifts labelled $\mathrm{2e}$ and $\mathrm{2d}$ in $G$, respectively, so we may assume that $t$ lies in either $F$-conjugacy class 2d or 2f. Now, $|C_G(t)|<|F|$ and $|C_G(t)|$ is divisible by $2^{34}$ whereas $|F|$ is not, so $C_G(t)\setminus F\ne\emptyset$. Thus $\langle C_G(t),F\rangle=G$. Since $\GG(t^F)$ is connected, the result now follows from Lemma~\ref{lem:centralizermaximal}.

Finally, suppose that $|Z(G)|=4$. Let $\sigma$ be an outer automorphism of order $3$. Then $C_G(\sigma)$ has a subgroup $O:=\mathrm{O}_{10}^-(2)$; by checking all possible class fusions in {\sf GAP} we see that $O$ has a class of involutions fusing to the $G$-class 2h of 2B-lifts. Thus, we may find an element $t\in\mathrm{2h}$ which is a $\mathrm{2B}$-lift centralized by $\sigma$. Now, the three central involutions $z_1,z_2,z_3$ are permuted cyclically by $\sigma$, so $tz_1,tz_2,tz_3$ are conjugate via $\sigma$. Therefore, it suffices to show that $\GG(t^G)$ and $\GG((tz_1)^G)$ are connected; without loss of generality $(tz_1)^G=\mathrm{2i}$. Consider the maximal subgroup $M:=2\times2\ns\F_4(2)<G$ --- we check class fusion in {\sf GAP} confirming that the $M$-conjugacy class $\mathrm{2d}$ fuses to the $G$-conjugacy class $\mathrm{2i}$, and $n_M(\mathrm{2d,2d,2d})=270$. Similarly, the $M$-conjugacy classes labelled $\mathrm{2e}$ and $\mathrm{2f}$ fuse to the $G$-conjugacy class labelled $\mathrm{2h}$, and $n_M(\mathrm{2e,2e,2f})=32$. Since the centralizers in $M$ of elements of $M$-classes 2d and 2e are maximal, we deduce that $M\leq N_G(\Lambda(t))$. As above, $\langle C_G(t),M\rangle =G$, so the result follows from Lemma~\ref{lem:connectedstabilizer}.
\end{proof}

We end this paper by describing how to read Table~\ref{tbl:exco}. Each relevant group $G$ occupies a row, with the group name appearing in the leftmost column. The classes $t^G$ of $p$-elements (where $p\mid |Z(G)|$) for which the graph $\GG(t^G)$ is disconnected appear in the second column; should no such class exist, `--' will appear in the second column, and the remaining entries will be left blank. For the remaining columns, information will appear in the order imposed by the list in the second column. Finally, in the case that the component group of $t$ is $\langle t\rangle$, we write `--' in place of the component stabilizer, as this is either the centralizer of $t$ or the normalizer of $\gen{t}$ in $G$.
{\small\section*{Declaration of AI use}
\noindent{A large language model (ChatGPT 5.6 Plus) was used during the revision process to audit our work for any errors. It did not suggest repairs.}}
\begin{sidewaystable}
\centering
\begin{tabular}{|l|c|l|l|l|l|}
\hline
$G$                   & disconnected classes                                   & image in $\widetilde{G}$                                          & $|\Lambda(t)|$                              & $\Delta(t)$                     & $N_G(\Lambda(t))$                    \\\hline
$2\ns\mathrm{L}_3(4)$           & --                                    &                    &                  &            &          \\
$4a\ns\mathrm{L}_3(4)$          & --                                    &                    &                  &            &          \\
$4b\ns\mathrm{L}_3(4)$          & --                                    &                    &                  &            &          \\
$2^2\ns\mathrm{L}_3(4)$         & $\mathrm{2d}$                         & $\mathrm{2A}$      & 3                & $2^2$      & $[2^8].3$          \\
$(2\times 4)\ns\mathrm{L}_3(4)$ & $\mathrm{2d}$                         & $\mathrm{2A}$      & 6                & $2^3$ &$[2^9].3$               \\
$4^2\ns\mathrm{L}_3(4)$         & $\mathrm{2d}$                         & $\mathrm{2A}$      & 12               & $2^4$ &$[2^{10}].3$                \\
$2\ns\mathrm{U}_4(2)$           & $\mathrm{2b}$                         & $\mathrm{2A}$      & 1                & $2^1$ &--                \\
$2\ns\mathrm{U}_6(2)$           & $\mathrm{2b},\mathrm{2c},\mathrm{2f}$ & 2A, 2A, 2C         & 1, 1, 560        & $2^1,2^1,2^{10}$     &--, --, $2^{10}.\mathrm{L}_3(4)$\\
$2^2\ns\mathrm{U}_6(2)$         & 2d, 2e, 2f, 2g, 2l                     & 2A, 2A, 2A, 2A, 2C & 1, 1, 1, 1, 1120 & $2^1,2^1,2^1,2^1,2^{11}$      &--, --, --, --, $2^{11}.\mathrm{L}_3(4)$         \\
$2\ns\mathrm{S}_6(2)$           & 2c                                    & 2C                 & 56               & $2_+^{1+6}$ &$2_+^{1+6}.\mathrm{L}_3(2)$\\
$2\ns\mathrm{O}_8^+(2)$                  & 2d                                  & 2B                 & 1                & $2^1$                &--\\
$2^2\ns\mathrm{O}_8^+(2)$                & --                                    &                    &                  &     &                 \\
$2\ns\G_2(4)$                    & --                                    &                    &                  &             &        \\
$2\ns\F_4(2)$            & 2c, 2e            & 2A, 2B                                  & 1, 1                               & $2^1,2^1$              &--,--                  \\
$2\ns\twist\B_2(8)$        & 2b                                              & 2A                                              & 14                                       & $2^4$                     &$[2^7].7$                       \\
$2^2\ns\twist\B_2(8)$      & 2d                                              & 2A                                              & 28                                       & $2^5$                      &$[2^8].7$                      \\
$2\ns\twist\E_6(2)$        & 2c                                              & 2A                                              & 1                                        & $2^1$                &--                            \\
$2^2\ns\twist\E_6(2)$      & 2e, 2f, 2g                                      & 2A, 2A, 2A                                      & 1, 1, 1                                  & $2^1, 2^1, 2^1$        &--,--,--                          \\
$3\ns\mathrm{Alt}(6)$               & $\mathrm{3c,3d}$                                & $\mathrm{3A},\mathrm{3B}$                       & 6                                        & $3^2$             &$3^2.\Sym(3)$                               \\
$3\ns\mathrm{Alt}(7)$               & $\mathrm{3c,3d}$                                & $\mathrm{3A},\mathrm{3B}$                       & 6, 6                                        & $3^2$, $3^2$ &$3^2.\Sym(3)$, $3^2.\Sym(4)$                                            \\\hline
$3a\ns\mathrm{U}_4(3)$  & 3c, 3d, 3e, 3f,
                & 3A, 3A, 3A, 3B,
                  & 2, 1, 1, 30,
               & $3^1, 3^1, 3^1, 3^5$, &--, --, --, $3^5.\Alt(6)$
        \\&3g, 3h, 3i, 3j &3B, 3B, 3C, 3D &1, 1, 90, 18 &$3^1, 3^1, 3^5, 3^3$ &--, --,$3^5.\Alt(6)$, $[3^6].2$ \\\hline
$3b\ns\mathrm{U}_4(3)$  & 3c, 3d, 3e, 3f, 3g, 3h                          & 3A, 3A, 3A, 3B, 3C, 3D                          & 2, 1, 1, 90, 90, 72                      & $3^1, 3^1, 3^1, 3^5, 3^5, 3^4$     &--, --, --, $3^5.\Alt(6)$,\\&&&&&$3^5.\Alt(6)$, $[3^6].[2^3]$            \\\hline
$3^2\ns\mathrm{U}_4(3)$ & 3i, 3j, $\dots$, 3q, 3r, &

3A, 3A, $\dots$, 3A, 3B,
& 2, 1, $\dots$, 1, 90,
& $3^1,3^1,\dots,3^1,3^6$,&--, --,$\dots$, --, $3^6.\Alt(6)$\\
&3s, 3t, 3u, 3v, 3w, 3x &
3B, 3B, 3C, 3C, 3C, 3D &
1, 1, 1, 1, 90, 54 &
$3^1,3^1,3^1,3^1,3^6,3^4$ &--, --, --, --, $3^6.\Alt(6)$, $[3^7].2$\\\hline
$3\ns\mathrm{O}_7(3)$            & 3d, 3e, 3f, 3g, 3j, 3k     & 3B, 3B, 3B, 3C, 3F, 3G                          & 72, 1, 1, 270, 1296,                & $3^6, 3^1, 3^1, 3^6, 3^{2+6}$,       & $3^6.\mathrm{S}_4(3).2$, --, --,\\&& &2592&$ 3^{2+6}$& $3^6.\mathrm{S}_4(3).2$, $3^{2+6}.[2^5].[3^2].2$,\\&&&&& $3^{2+6}.[2^5].[3^2].2$   \\\hline
$3\ns\G_2(3)$            & 3f, 3g                                          & 3D, 3E                                          & 54, 54                                       & $3^4$, $3^4$& $[3^7].[2^2]$, $[3^7].[2^2]$\\\hline
\end{tabular}
\caption{Classes of exceptional covers of simple groups for which $\GG(\mathcal{C})$ is disconnected.}
\label{tbl:exco}

\end{sidewaystable}

\bibliographystyle{amsplain}
\bibliography{references}

@article {AschSeg92,
    AUTHOR = {Aschbacher, M.\ and Segev, Y.},
     TITLE = {The uniqueness of groups of {L}yons type},
   JOURNAL = {J.\ Amer.\ Math.\ Soc.},
  FJOURNAL = {Journal of the American Mathematical Society},
    VOLUME = {5},
      YEAR = {1992},
    NUMBER = {1},
     PAGES = {75--98},
      ISSN = {0894-0347,1088-6834},
   MRCLASS = {20D08},
MRREVIEWER = {Gernot\ Stroth},
       DOI = {10.2307/2152751},
       URL = {https://doi-org.ezproxy.st-andrews.ac.uk/10.2307/2152751},
}

@article {IranJafar08,
    AUTHOR = {Iranmanesh, A.\ and Jafarzadeh, A.},
     TITLE = {On the commuting graph associated with the symmetric and
              alternating groups},
   JOURNAL = {J.\ Algebra Appl.},
  FJOURNAL = {Journal of Algebra and its Applications},
    VOLUME = {7},
      YEAR = {2008},
    NUMBER = {1},
     PAGES = {129--146},
      ISSN = {0219-4988,1793-6829},
   MRCLASS = {20B30 (05C25)},
MRREVIEWER = {Michael\ Giudici},
       DOI = {10.1142/S0219498808002710},
       URL = {https://doi-org.ezproxy.st-andrews.ac.uk/10.1142/S0219498808002710},
}

@article {BBHR07,
    AUTHOR = {Bates, C.\ and Bundy, D.\ and Hart, S.\ and Rowley, P.},
     TITLE = {Commuting involution graphs for sporadic simple groups},
   JOURNAL = {J.\ Algebra},
  FJOURNAL = {Journal of Algebra},
    VOLUME = {316},
      YEAR = {2007},
    NUMBER = {2},
     PAGES = {849--868},
      ISSN = {0021-8693,1090-266X},
   MRCLASS = {20D08 (05C25)},
MRREVIEWER = {Andrew\ Woldar},
       DOI = {10.1016/j.jalgebra.2007.04.019},
       URL = {https://doi-org.ezproxy.st-andrews.ac.uk/10.1016/j.jalgebra.2007.04.019},
}

@ARTICLE{brauerfowler1955,
  AUTHOR =       "Brauer, R.\ and Fowler, K.",
  TITLE =        "On Groups of Even Order",
  JOURNAL =      "Ann.\ Math.",
  YEAR =         "1955",
  volume =       "62",
  pages =        "565--583",
}

@article{GG25,
 author = {Gill, N.\ and Guillot, P.},
 title = {The binary actions of alternating groups},
 journal = {Confluentes Math.},
 volume = {17},
 pages = {73--89},
 year = {2025},
}

@article {GG252,
    AUTHOR = {Gill, N.\ and Guillot, P.},
     TITLE = {The binary actions of simple groups with a single conjugacy
              class of involutions},
   JOURNAL = {J.\ Group Theory},
  FJOURNAL = {Journal of Group Theory},
    VOLUME = {28},
      YEAR = {2025},
    NUMBER = {1},
     PAGES = {215--240},
      ISSN = {1433-5883,1435-4446},
   MRCLASS = {20D05 (05C25 20E45)},
MRREVIEWER = {Qinhui\ Jiang},
       DOI = {10.1515/jgth-2024-0066},
       URL = {https://doi-org.ezproxy.st-andrews.ac.uk/10.1515/jgth-2024-0066},
}

@article {GGL25,
    AUTHOR = {Gill, N.\ and Guillot, P.\ and Liebeck, M.W.},
     TITLE = {The binary actions of simple groups of {L}ie type of
              characteristic 2},
   JOURNAL = {Pacific J.\ Math.},
  FJOURNAL = {Pacific Journal of Mathematics},
    VOLUME = {336},
      YEAR = {2025},
    NUMBER = {1--2},
     PAGES = {113--135},
      ISSN = {0030-8730,1945-5844},
   MRCLASS = {20D06 (20E45)},
       DOI = {10.2140/pjm.2025.336.113},
       URL = {https://doi-org.ezproxy.st-andrews.ac.uk/10.2140/pjm.2025.336.113},
}

@article{delvallegill2025+,
    author = {del Valle, C.\ and Gill, N.},
    title = {The binary actions of sporadic groups},
    journal = {arXiv preprint.}

}

@article{dVGL26,
    author = {del Valle, C.\ and Gill, N.\ and Liebeck, M.W.},
    title = {The binary actions of low rank groups of {L}ie type},
    journal = {in preparation.}

}

@incollection {cherlin2000,
    AUTHOR = {Cherlin, G.},
     TITLE = {Sporadic homogeneous structures},
 BOOKTITLE = {The {G}elfand {M}athematical {S}eminars, 1996--1999},
    SERIES = {Gelfand Math.\ Sem.},
     PAGES = {15--48},
 PUBLISHER = {Birkh\"auser Boston, Boston, MA},
      YEAR = {2000},
      ISBN = {0-8176-4013-4},
   MRCLASS = {03C13 (03C60 20B05)},
MRREVIEWER = {H.\ Dugald\ Macpherson},
}

@article {Lach84,
    AUTHOR = {Lachlan, A.H.},
     TITLE = {On countable stable structures which are homogeneous for a
              finite relational language},
   JOURNAL = {Israel J.\ Math.},
  FJOURNAL = {Israel Journal of Mathematics},
    VOLUME = {49},
      YEAR = {1984},
    NUMBER = {1--3},
     PAGES = {69--153},
      ISSN = {0021-2172},
   MRCLASS = {03C45 (03C15)},
MRREVIEWER = {B.\ I.\ Zil\cprime ber},
       DOI = {10.1007/BF02760647},
       URL = {https://doi-org.ezproxy.st-andrews.ac.uk/10.1007/BF02760647},
}

@misc{codesup,
    Author = {Craven, D.A.\ and del Valle, C.\ and Parker, C.W.},
    title = {Supplementary materials for `{T}he {G}ill--{G}uillot graph for sporadic and related groups'},
    note = {{\sc {Magma}} code, 2026, hosted with the paper on the arXiv.}
}

@book {gorenstein,
    AUTHOR = {Gorenstein, D.},
     TITLE = {Finite groups},
   EDITION = {Second},
 PUBLISHER = {Chelsea Publishing Co., New York},
      YEAR = {1980},
     PAGES = {xvii+519},
      ISBN = {0-8284-0301-5},
   MRCLASS = {20-02 (20Dxx)},
}

@misc{ CTblLib1.3.11, 
author = {Breuer, T.}, 
title = {The \textsf{GAP} {C}haracter {T}able {L}ibrary, {V}ersion 1.3.11}, 
month ={May}, 
year = {2025},
note = {\url{http://www.math.rwth-aachen.de/~Thomas.Breuer/ctbllib}, \textsf{GAP} package} }

@article {magma,
    AUTHOR = {Bosma, W.\ and Cannon, J.\ and Playoust, C.},
     TITLE = {The {M}agma algebra system. {I}. {T}he user language},
      NOTE = {Computational algebra and number theory (London, 1993)},
   JOURNAL = {J.\ Symbolic Comput.},
  FJOURNAL = {Journal of Symbolic Computation},
    VOLUME = {24},
      YEAR = {1997},
    NUMBER = {3-4},
     PAGES = {235--265},
      ISSN = {0747-7171},
   MRCLASS = {68Q40},
       DOI = {10.1006/jsco.1996.0125},
       URL = {http://dx.doi.org/10.1006/jsco.1996.0125},
}

@book {atlas,
    AUTHOR = {Conway, J.H.\ and Curtis, R.T.\ and Norton, S.P.\ and Parker, R.A.\ and Wilson, R.A.},
     TITLE = {{$\mathbb{ATLAS}$} of finite groups},
      NOTE = {Maximal subgroups and ordinary characters for simple groups,
              With computational assistance from J.G.\ Thackray},
 PUBLISHER = {Oxford University Press, Eynsham},
      YEAR = {1985},
     PAGES = {xxxiv+252},
      ISBN = {0-19-853199-0},
   MRCLASS = {20D05 (20-02)},
MRREVIEWER = {R.\ L.\ Griess},
}

@article {Bender,
    AUTHOR = {Bender, H.},
     TITLE = {Transitive {G}ruppen gerader {O}rdnung, in denen jede
              {I}nvolution genau einen {P}unkt festl\"a\ss t},
   JOURNAL = {J.\ Algebra},
  FJOURNAL = {Journal of Algebra},
    VOLUME = {17},
      YEAR = {1971},
     PAGES = {527--554},
      ISSN = {0021-8693},
   MRCLASS = {20.25},
MRREVIEWER = {M.\ E.\ Harris},
       DOI = {10.1016/0021-8693(71)90008-1},
       URL = {https://doi-org.ezproxy.st-andrews.ac.uk/10.1016/0021-8693(71)90008-1},
}

@article {AschbacherExistence,
    AUTHOR = {Aschbacher, M.},
     TITLE = {A condition for the existence of a strongly embedded subgroup},
   JOURNAL = {Proc.\ Amer.\ Math.\ Soc.},
  FJOURNAL = {Proceedings of the American Mathematical Society},
    VOLUME = {38},
      YEAR = {1973},
     PAGES = {509--511},
      ISSN = {0002-9939,1088-6826},
   MRCLASS = {20D25 (05C25)},
MRREVIEWER = {Henry\ S.\ Leonard, Jr.},
       DOI = {10.2307/2038941},
       
}

@article {MorPar13,
    AUTHOR = {Morgan, G.L.\ and Parker, C.W.},
     TITLE = {The diameter of the commuting graph of a finite group with
              trivial centre},
   JOURNAL = {J.\ Algebra},
  FJOURNAL = {Journal of Algebra},
    VOLUME = {393},
      YEAR = {2013},
     PAGES = {41--59},
      ISSN = {0021-8693,1090-266X},

}

@book {GLS2,
    AUTHOR = {Gorenstein, D.\ and Lyons, R.\ and Solomon, R.},
     TITLE = {The classification of the finite simple groups. {N}umber 2.
              {P}art {I}. {C}hapter {G}},
    SERIES = {Mathematical Surveys and Monographs},
    VOLUME = {40.2},
 PUBLISHER = {American Mathematical Society, Providence, RI},
      YEAR = {1996},
     PAGES = {xii+218},

}

@book {GLS3,
    AUTHOR = {Gorenstein, D.\ and Lyons, R.\ and Solomon, R.},
     TITLE = {The classification of the finite simple groups. {N}umber 3.
              {P}art {I}. {C}hapter {A}},
    SERIES = {Mathematical Surveys and Monographs},
    VOLUME = {40.3},
 PUBLISHER = {American Mathematical Society, Providence, RI},
      YEAR = {1998},
     PAGES = {xvi+419},
}

@article {dlpp2024un,
    AUTHOR = {Dietrich, H. and Lee, M. and Pisani, A. and
              Popiel, T.},
     TITLE = {Explicit construction of the maximal subgroups of the
              {M}onster},
   JOURNAL = {J. Algebra},
    VOLUME = {689},
      YEAR = {2026},
     PAGES = {862--895},

}

\end{document}